\documentclass[11pt,letterpaper]{amsart}
\usepackage{amssymb}
\usepackage{mathrsfs, tikz}
\usepackage[all,cmtip]{xy}
\usepackage{pstricks}

\usepackage{ulem}
\usepackage{comment}

\usepackage{graphpap,color}

\definecolor{cof}{RGB}{219,144,71}
\definecolor{pur}{RGB}{186,146,162}
\definecolor{greeo}{RGB}{91,173,69}
\definecolor{greet}{RGB}{52,111,72}

\numberwithin{equation}{section}

\newtheorem{thm}{Theorem}[section]
\newtheorem{defn}[thm]{Definition}
\newtheorem{prop}[thm]{Proposition}
\newtheorem{prop-defn}[thm]{Proposition-Definition}

\newtheorem{lemma}[thm]{Lemma}
\newtheorem{cor}[thm]{Corollary}

\newcommand{\PGL}{{\rm PGL}}

\newcommand{\Id}{{\rm Id}}
\newcommand{\Gr}{{\rm Gr}}

\newcommand{\CD}{\xymatrix@R=1pc@C=1pc}
\newcommand{\CDR}{\xymatrix@R=1pc}
\newcommand{\CDC}{\xymatrix@C=1pc}

\def\cB{{\mathcal B}}

\def\cD{{\mathcal D}}

\def\cE{{\mathcal E}}

\def\cL{{\mathcal L}}

\def\cO{{\mathcal O}}

\def\cV{{\mathcal V}}

\def\cZ{{\mathcal Z}}

\def\sF{{\mathscr F}}

\def\sR{\mathscr{R}}
\def\sV{\mathscr{V}}
\def\tsV{\widetilde{\mathscr{V}}}

\def\fG{\mathfrak{G}}

\def\fL{\mathfrak{L}}

\def\fV{\mathfrak{V}}

\def\fr{\mathfrak{r}}

\def\fr{\mathfrak{r}}
\def\fs{\mathfrak{s}}

\def\NN{{\mathbb N}}
\def\PP{{\mathbb P}}

\def\GG{{\mathbb G}}

\def\ZZ{{\mathbb Z}}

\def\AA{{\mathbb A}}

\def\Ga{{\Gamma}}
\def\tGa{{\widetilde{\Gamma}}}

\def\vt{{\varTheta}}

\def\vt{{\vartheta}}

\def\si{{\sigma}}

\def\var{{\rm Var}}

\def\lex{{\rm lex}}

\def\zero{{=0}}
\def\one{{=1}}

\def\tZ{{\widetilde Z}}

\def\lra{\longrightarrow}

\def\kk{{\bf k}}

\def\bp{{\bf p}}

\def\bU{{\bf U}}

\def\bz{{\bf z}}

\def\um{{\underbar {{\mu}}}}
\def\um{{\underbar {\it m}}}
\def\uu{{\underbar {\it u}}}
\def\uv{{\underbar {\it v}}}

\def\pl{{\hbox{Pl\"ucker}}}

 \def\2{{\rm I\!I}}

\def\-{{\setminus}}

\def\vp{{\varpi}}
\def\vr{{\varrho}}

\def\vi{{\varphi}}

\def\cBkd1{{\cB_{[k]}^{d_\vr=1}}}
\def\cBd1{{\cB^{d_\vr=1}}}

\def\modcBkd1{{\mod \; \cB_{[k]}^{d_\vr=1}}}

\def\Jac{{\rm Jac}}

\def\sgn{{\rm sgn}}

\def\rk{{\rm rank \;}}

\def\ngv{{\rm ngv}}

\def\ngvr{{\rm ngv1}}
\def\ngvh{{\rm ngv2}}

\def\gov{{\rm gov}}

\def\mh{{\hbox{\scriptsize ${\rm mh}$}}}

\def\de{\delta}

\def\tsR{{\widetilde{\sR}}}

\def\up{{\Upsilon}}

\def\ud{{\underbar  d}}

\makeatletter
\newcommand*\bigcdot{\mathpalette\bigcdot@{.7}}
\newcommand*\bigcdot@[2]{\mathbin{\vcenter{\hbox{\scalebox{#2}{$\m@th#1\bullet$}}}}}
\makeatother

\makeatletter
\def\@tocline#1#2#3#4#5#6#7{\relax
  \ifnum #1>\c@tocdepth % then omit
  \else
    \par \addpenalty\@secpenalty\addvspace{#2}%
    \begingroup \hyphenpenalty\@M
    \@ifempty{#4}{%
      \@tempdima\csname r@tocindent\number#1\endcsname\relax
    }{%
      \@tempdima#4\relax
    }%
    \parindent\z@ \leftskip#3\relax \advance\leftskip\@tempdima\relax
    \rightskip\@pnumwidth plus4em \parfillskip-\@pnumwidth
    #5\leavevmode\hskip-\@tempdima
      \ifcase #1
       \or\or \hskip 1em \or \hskip 2em \else \hskip 3em \fi%
      #6\nobreak\relax
    \hfill\hbox to\@pnumwidth{\@tocpagenum{#7}}\par% <---- \dotfill -> \hfill
    \nobreak
    \endgroup
  \fi}
\makeatother

\begin{document}

\title[Universal Equations and Characteristic-free Resolution]
{Introduction to Universal Equations and Characteristic-free Resolution}

\date{}
\author{Yi Hu}
%\address{Department of Mathematics, University of Arizona, USA.}
%\email{yhu@math.arizona.edu}

\maketitle

\begin{abstract} 
Recently, we announced a proof of a characteristic-free resolution of singularities. This is a survey article on the proof. We motivate the approach, 
 describe the new ideas, and introduce the proof in the  work.
 \end{abstract}

\maketitle

%\tableofcontents

%\section{Introduction}
\medskip

When an arbitrary singularity  over $\ZZ$
 is randomly placed in an  affine space, its defining equations
are necessarily arbitrary. Traditionally, 
it has been believed that there is no standard embedding
with standardized relations. 
We show that this belief is false by exhibiting a standard embedding  of every singularity
and,  more importantly, prove that a birational transform of any singularity admits an embedding 
whose defining relations consist of only {\it square-free
binomials and linear relations}. Equally importantly, these  simple(st) relations are 
{\it structurally well organized},  and  {\it universal} for all singularities. 

Due to the above universality and structural simplicity, 
we  prove that  for any singular integral affine variety $X$,
there exist a smooth scheme $\widetilde{Y}$ and a projective birational   morphism from
$\widetilde{Y}$ onto $Y$, followed by a  smooth morphism 
from $Y$ onto $X$.  

In our approach, we neither restrict to any  singular variety nor fix the characteristic. 
Instead, we construct a  {\it universal blowup process} that {\it simultaneously} resolves all possible singularities, and,
 our method is entirely characteristic-free.

There are two closely related themes in our approach. First,
all singularities admit structurally well organized 
universal equations. Second, we can  compute Jacobian matrices of  their proper transforms under our carefully constructed universal blowups.

The goal of this introduction is to make the idea and approach of \cite{Hu2025}  
transparent. Once it is understood, 
checking  its details is routine.\footnote{For Chapters 
5, 6, and 8 of \cite{Hu2025}, the reader should read \cite{Hu2026b}.}

\section{Lafforgue's version of Mn\"ev's universality theorem}

%Consider an arbitrary affine variety $X$  in $\AA^k_\ZZ$.
\begin{comment}
defined by a set of polynomial relations $\langle f_1, \cdots, f_r \rangle$.
arbitrary in that every $f_i$ is in arbitrary form, and given $f_i$, there are no (combinatorial) patterns
to predict the next polynomial $f_{i+1}$. The smoothness of $X$ over a field, despite the existence of various criteria, is tested by the Jacobian matrix ${\rm Jac}  (f_1, \cdots, f_r)$. 
Over $p>0$, there is no reason
$x^p$ does not occur in $f_i$, and there is no reason  blowups can eliminate $x^p$ in its final local equations $(g_1, \cdots, g_{r'})$, thereby there is no basis to conclude 
${\rm Jac} (g_1, \cdots, g_{r'})$  being full rank everywhere if one's process starts with
the arbitrary relations $\langle f_1, \cdots, f_r \rangle$. 
\end{comment}
There seem two traditional believes: %\begin{itemize} \item
 (1) to describe an arbitrary singularity, one needs a set of arbitrary 
polynomial relations. % $\langle f_1, \cdots, f_r \rangle$;
 (2) to resolve singularities, one needs to focus on a  variety, introduce bounded invariants,
and then blow up to strictly improve these invariants. %\end{itemize}
We show  that 
\begin{itemize}
\item  (1) is false and
\item   (2) is unnecessary (perhaps themselves obstacles over $p>0$).
\end{itemize}

\subsection{Atomic equations for affine varieties} $\ $

Given an  affine variety $X/\ZZ$  in $\AA^k$, 
defined by relations $\langle f_1, \cdots, f_r \rangle$,
one can decode every $f_i$ into a set of simplest relations of four types in some $\AA^m$
 \begin{equation}\label{atom} %\xymatrix{
X_\gamma=X_\alpha X_\beta \\
%X_\gamma=X_\alpha + X_\beta \\
%X_\gamma=X_\alpha+1}.
\end{equation}
$$X_{\gamma'}=X_{\alpha'}+ X_{\beta'} $$
$$X_{\gamma''}=X_{\alpha''}+1$$
$$X_a=X_b.$$
The process can be understood from the following example of \cite{LV12}
$${{\ZZ[x_1,x_2,x_3]}
\over {x_1x_2+x_3^2-2}}={{\ZZ[x_1,x_2,x_3,x_4,x_5,x_6,x_7,x_8,x_9,x_{10}] }\over 
{x_4-x_1x_2,x_5-x_3, x_6-x_3x_5, x_7-x_4-x_6, x_8-1,x_9-1, x_{10}-x_8-x_9, x_7-x_{10}}}.
$$
We call \eqref{atom} the atomic equations of $X$.\footnote{Though each  atomic equation is  simple,
 altogether they are random with no structures. This is not of much use unless
we find {\it molecular equations} for $X$,
i.e., simple equations formed together with structures.}

\subsection{Singularities as configuration spaces} $\ $

Consider $X \subset \AA^m$  defined by atomic equations of \eqref{atom}.
We can place all variables $X_\alpha$ as points in $\PP^2$,
subject to  position constraints according to these relations.
To do this, we choose a fixed point $0 \in \PP^2$ and a line $\ell_\infty$.
For any variable $X_\alpha$, we choose a general free point $\infty_\alpha \in \ell_\infty$ and
 $1_\alpha$ in the line $\overline{0\infty_\alpha}$ connecting $0$ and $\infty_\alpha$. We then  can put $X_\alpha$ in the line  $\overline{0\infty_\alpha}$, using $M_{0,4}$ or cross-ratio $[0,1_\alpha, \infty_\alpha, X_\alpha]$.
For any equation $e$ of the first three types in \eqref{atom}, say,
$X_\gamma=X_\alpha + X_\beta$,  we choose a general free point $\infty_e \in \ell_\infty$,
$1_e \in \overline{0\infty_e}$,
and  let $X_e \in \overline{0\infty_e}$ be determined by Figure 1.5 of \cite{La03}.
Similar treatments for $X_{\gamma'}=X_{\alpha'}+ X_{\beta'} $ (Figure 1.6 of \cite{La03}) and
$X_{\gamma''}=X_{\alpha''}+1$  (Figure 1.7 of \cite{La03}).
This  establishes a relation between $X$ and a configuration space $C_{\ud}^{3,n}$
of $n$ points in $\PP^2$, subject to position constrains  by a matroid $\ud$,
where the matroid $\ud$ here is a rank function 
$$r: 2^{[n]} \lra \NN$$ 
where $[n]=\{1,\cdots, n\}$, satisfying $r(\emptyset)=0$, $r([n])=3$, and for any $I, J \subset [n]$
\begin{enumerate}
\item $r(I) \le |I|$;
\item $r(I \cup J)+ r(I \cap J) \le r(I) + r(J)$;
\item $r (I \cup \{x\}) \le r(I) +1$.
\end{enumerate}
This is called a matroid of rank 3 on the set $[n]$.
Under this terminology, 
$$C^{3,n}_\ud =\{ \bp=(p_1, \cdots, p_n) \in (\PP^2)^n \mid  r_\bp =r\},$$ where
$r_\bp (I)=\dim {\rm span}\{v_i \mid i \in I\}, \forall \; I \subset [n]$, 
with $v_i$ being a lift of $p_i$ in the vector space.

\begin{thm}{\rm (Lafforgue's version of Mn\"ev's universality theorem, I)}
Given any affine variety $X/\ZZ$, there exist integers $n$, $r$,  a matroid $\ud$
of rank 3 on the set $[n]$, and the following diagram
$$ \xymatrix{
C_\ud^{3,n} \ar[r] \ar[rrd]  &  C_\ud^{3,n}/\PGL_3  \cong U  \ar[rd] \ar @{^{(}->}[r]   &X \times \AA^r \ar[d] \\
& & X
}
$$
such that $\PGL_3$ acts freely on $C_\ud^{3,n}$, $U$ is open in $X \times \AA^r$,
and the composition morphism $U \lra X$ is onto. In particular, the surjective morphism
$C_\ud^{3,n} \lra X$ is smooth.
\end{thm}
Here, the  free $\PGL_3$-action removes the choice of $0$ and $\ell_\infty$, and
$r$ counts the number of auxiliary  free points 
$(1_\alpha, \infty_\alpha, 1_e, \ldots)$. 
%The reason $U$ is open because the auxiliary free points are required to be distinct, for example.

\subsection{Singularities as strata of  Grassmannians} $\ $

To connect with Grassmannian,  we write points $p_1, \cdots, p_n$ of $\PP^2$
as  matrices
$$\left\{
\left(
\begin{array}{cccccccccc}
x_1 & x_2 & \cdots & x_n \\
y_1 & y_2 & \cdots & y_n \\
z_1 & z_2 & \cdots & z_n 
\end{array}
\right) \right\}
$$
The group ${\rm GL}_3$ acts from the left and $\GG_m^n$ acts  coordinate-wise.
If we divide $\GG_m^n$  first, we obtain $(\PP^2)^n$ with the residual $\PGL_3$ action;
if we divide ${\rm GL}_3$ action first, we obtain the Grassmannian $\Gr^{3,n}$ of 3-dimensional subspaces in  $n$-dimensional vector space with the residual 
 $\GG_m^n/\GG_m$ action. This gives rise to  the equality of two orbit spaces
$$[(\PP^2)^n/\PGL_3] = [\Gr^{3,n}/(\GG_m^n/\GG_m)],$$
called the Gelfand-MacPherson correspondence. Using it, we obtain
$$  C_\ud^{3,n}/\PGL_3  = \Gr^{3,n}_\ud/(\GG_m^n/\GG_m), \;\; \hbox{where}$$
 $$\Gr^{3,n}_\ud =\{E \in \Gr^{3,n} \mid \dim E \cap E_I = r(I), \; \forall \; I \in [n]\}$$
where $E_I$ is the coordinate subspace spanned by the basis vectors $e_i, i \in I$.

\begin{thm}{\rm (Lafforgue's version of Mn\"ev's universality theorem, II)}
Given any  affine variety $X/\ZZ$, there exist integers $n$, $r$,  a matroid $\ud$
of rank 3 on the set $[n]$,  and the following diagram
$$ \xymatrix{
\Gr_\ud^{3,n} \ar[r] \ar[rrd]  &  \Gr_\ud^{3,n}/(\GG_m^n/\GG_m)  \cong U  \ar[rd] \ar @{^{(}->}[r]   &X \times \AA^r \ar[d] \\
& & X
}
$$
such that $(\GG_m^n/\GG_m)$ acts freely on $\Gr_\ud^{3,n}$, $U$ is open in $X \times \AA^r$,
and the  morphism $U \lra X$ is onto. In particular, the surjective morphism
$\Gr_\ud^{3,n} \lra X$ is smooth.
\end{thm}

\subsection{Molecular equations for singularities}
Consider the $\pl$ embedding 
$$\Gr^{3,n} \lra \PP( \wedge^3 \kk^n)$$  with
the $\pl$ coordinates $[p_{ijk}]_{1 \le i<j<k \le n}$, where $\kk$ is a field.

\begin{prop}\label{pre-molecular} {\rm (Lafforgue)} Given any matroid $\ud$ of rank 3 on the set $[n]$, there
exists a subset $\Gamma$ of the set of all $\pl$ variables such that
$\Gr^{3,n}_\ud$ as a subset of $ \PP( \wedge^3 \kk^n)$ is defined by 
$$\hbox{all $\pl$ relations $F$, together with}$$
$$p_{ijk}=0, \; \forall \; p_{ijk} \in \Gamma \ \; \hbox{and} \ \; 
p_{ijk} \ne 0, \; \forall \; p_{ijk} \notin \Gamma.$$
\end{prop}
%For an explicit expression of 
The subset $\Gamma$ is determined by the matroid $\ud$ (Proposition p4, \cite{La03}.
Proposition 9.1, \cite{Hu2025}).

To define $\Gr^{3,n}_\ud$, not all $\pl$ relations are needed, as many are redundant.
Indeed, $\Gr^{3,n}_\ud$ is contained in some affine chart of $ \PP( \wedge^3 \kk^n)$,
up to permutation, we may assume this chart is $\bU =(p_{123}\ne 0)$ with affine coordinates
$$x_{ijk}=p_{ijk}/p_{123}, \; \forall \; ijk \ne 123.$$

\begin{thm}\label{molecular} {\rm (\cite{Hu2025})}
 Given any matroid $\ud$ of rank 3 on the set $[n]$, up to permutation,
 we assume that $\Gr^{3,n}_\ud$ is contained in the affine chart $\bU =(p_{123}\ne 0)$.
 Then, there
exists a subset $\Gamma \subset \{x_{ijk} \mid  ijk \ne 123\}$ such that
$\Gr^{3,n}_\ud$ as a subset of $\bU$ is defined by 
\begin{eqnarray}\label{tour: all primary pl}
%\hbox{de-homogenized primary $\pl$ relations for $\bU \cap \Gr^{3,E}$:}  \\
F_{1uv}:  \; x_{1uv}-x_{12u}x_{13v} + x_{13u}x_{12v}, \nonumber
\label{rk0-1uv}\\
F_{2uv}:  \; x_{2uv}-x_{12u}x_{23v} + x_{23u}x_{12v}, \nonumber
\label{rk0-2uv} \\
F_{3uv}: \; x_{3uv}-x_{13u}x_{23v} + x_{23u}x_{13v} , \nonumber
\label{rk0-3uv}\\
F_{abc}: \; x_{abc}-x_{12a}x_{3bc} + x_{13a}x_{2bc} -x_{23a}x_{1bc}, \label{rk1-abc} \nonumber
\end{eqnarray}
where $3<u < v \le n$ and $3<a<b<c \le n$, together with
$$x_{ijk}=0, \; \forall \; x_{ijk} \in \Gamma \ \; \hbox{and} \ \; 
x_{ijk} \ne 0, \; \forall \; x_{ijk} \notin \Gamma.$$
\end{thm}
\begin{proof}  This follows by applying Proposition \ref{pre-molecular} and
Proposition 3.6 of  \cite{Hu2025}.
\end{proof}

\begin{defn} \label{molecular2} For any  $\Gamma \subset \{x_{ijk} \mid  ijk \ne 123\}$, we let 
$Z_\Ga \subset \bU$ be defined by 
\begin{eqnarray}\label{molecular-eqs} \nonumber
F_{1uv}, F_{2uv}, F_{3uv}, F_{abc}, \; \; \forall \;
3<u < v \le n, \; 3<a<b<c \le n, \\  \nonumber
x_{ijk}=0, \; \forall \; x_{ijk} \in \Gamma. \;\;\; \;\;\; \;\;\; \;\;\; \;\;\; \;\;\; \;\;\; \;\;\; \;\;\; \;\;\;
\end{eqnarray}
\end{defn}
Then, $\Gr^{3,n}_\ud$ is an open subset of the affine variety $Z_\Ga$. To resolve 
$\Gr^{3,n}_\ud$, it suffices to resolve $Z_\Ga$. 
We may call the above equations 
 molecular equations of singularities,  as they are simple and 
structurally well organized.

\begin{defn}\label{led-via}
Every  of $$F_{1uv}, F_{2uv}, F_{3uv}, F_{abc}, \; \;
3<u < v \le n, \; 3<a<b<c \le n$$ has, respectively, 
the leading variable\footnote{The leading variables play 
leading r\^oles in our universal blowups and decisively 
enable us to compute Jacobian matrices of  relations. 
See \S\S \ref{foc-jac} and \ref{jjj}.}
$$x_{1uv}, x_{2uv}, x_{3uv}, x_{abc}, \; \;
3<u < v \le n, \; 3<a<b<c \le n.$$
\end{defn}

\section{Universal equations for all singularities}

\subsection{The universe $\sV$ for all singularities  and universal equations} $\ $

To construct our universal resolution, we separate the terms of every  $F$ in
$$\sF=\{F_{1uv}, F_{2uv}, F_{3uv}, F_{abc} \mid 3<u < v \le n, \; 3<a<b<c \le n \}.$$ 
Precisely, we want to replace them by binomials and linearized $\pl$ relations. 

To that end, for any $F \in \sF$,  expressed as
$$F=\sum_{s \in S_F} \sgn(s) x_{\uu_s}x_{\uv_s}, \;\; \hbox{for some index set $S_F$},$$
%where $s_F \in S_F$ gives the leading term $x_{\uu_F}$ 
with $x_{123}:=1$,  we introduce the projective space
$$\PP_F, \; \hbox{equipped with the homogeneous coordinates $[x_{(\uu_s,\uv_s)}]_{s \in S_F}$}.$$

We  then introduce the rational map 
$$\Theta: \bU \cap \Gr^{3, n}  \lra  \prod_{F \in \sF} \PP_F$$
$$[x_{ijk}]_{ijk \ne 123} \to  \prod_{F \in \sF} [x_{\uu_s}x_{\uv_s}]_{s \in S_F},$$
 let $\sV$ be the closure of the graph of $\Theta$, and obtain the diagram
$$ \xymatrix{
\sV \ar[d] \ar @{^{(}->}[r]   \ar[d] \ar @{^{(}->}[r]  &
\sR:= \bU  \times \prod_{F \in \sF} \PP_F \ar[d] \\
\bU \cap \Gr^{3, n}  \ar @{^{(}->}[r]    & \bU.}
$$
The left vertical morphism is birational.

\begin{defn}
The smooth scheme $\sR$ comes equipped with two kinds of variables
\begin{itemize}
\item we call $x_\uu$ (e.g., $x_{12u}$) a $\vp$-variable;
\item  we call $x_{(\uu,\uv)}$ (e.g., $x_{(12u,13v)}$) a $\vr$-variable.
\end{itemize}
\end{defn}

\begin{thm}\label{universal-eqs} {\rm (\cite{Hu2025})}
The scheme $\sV$ as a closed subset of $\sR$ is defined by 
\begin{itemize}
\item{\rm GL} {\rm  (Governing Linearized $\pl$ Relations)}
$$L_F: \sum_{s \in S_F} \sgn(s) x_{(\uu_s, \uv_s)}, 
\; \forall \; F =\sum_{s \in S_F} \sgn(s) x_{\uu_s}x_{\uv_s}\in \sF;$$
\item{\rm GB} {\rm (Governing Binomials)}
\\[-2em]
\begin{eqnarray}\label{GovB}\nonumber
F_{1uv}: 
x_{1uv}x_{(12u,13v)} - x_{12u}x_{13v} x_{(123,1uv)}, \; x_{1uv}x_{(13u,12v)}- x_{13u}x_{12v}x_{(123,1uv)}; \\
F_{2uv}: x_{2uv}x_{(12u,23v)} -x_{12u}x_{23v} x_{(123,2uv)}, \;  x_{2uv}x_{(23u,12v)}-x_{23u}x_{12v} x_{(123,2uv)}; \nonumber \\
F_{3uv}: x_{3uv}x_{(13u,23v)} -x_{13u}x_{23v}x_{(123,3uv)}, \;  x_{3uv}x_{(23u,13v)} -x_{23u}x_{12v}x_{(123,3uv)}; \nonumber\\
F_{abc}:  x_{abc}x_{(12a,3bc)}-x_{12a}x_{3bc} x_{(123,abc)},\;
 x_{abc}x_{(13a,2bc)}-x_{13a}x_{2bc} x_{(123,abc)},\nonumber \\
x_{abc}x_{(23a,1bc)} -x_{23a}x_{1bc}x_{(123,abc)}. \nonumber
\end{eqnarray}
\\[-3em]
\item{\rm NGB1}   {\rm (Non-governing Binomials, I)}
\\[-2em]
\begin{eqnarray}\label{GovB}\nonumber
F_{1uv}:  x_{12u}x_{13v} x_{(123,1uv)}- x_{13u}x_{12v}x_{(123,1uv)}; \\
F_{2uv}: x_{12u}x_{23v} x_{(123,2uv)} -x_{23u}x_{12v} x_{(123,2uv)}; \nonumber \\
F_{3uv}:  x_{13u}x_{23v}x_{(123,3uv)} -x_{23u}x_{12v}x_{(123,3uv)}; \nonumber\\
F_{abc}:  
x_{12a}x_{3bc}x_{3bc} x_{(123,abc)}  -x_{13a}x_{2bc} x_{(123,abc)},\nonumber \\
x_{12a}x_{3bc}x_{3bc} x_{(123,abc)} -x_{23a}x_{1bc}x_{(123,abc)} \\
x_{13a}x_{2bc} x_{(123,abc)} -x_{23a}x_{1bc}x_{(123,abc)}. \nonumber
\end{eqnarray}
\\[-3em]
\item{\rm NGB2} {\rm (Non-governing Binomials, II)}  For example, 
\\[-2em]
\begin{eqnarray}\nonumber
x_{1bc} x_{2b'c'} x_{(13a,2bc)} x_{(23a,1b'c')}-x_{2bc}  x_{1b'c'} x_{(23a,1bc)} x_{(13a,2b'c')} \\ \nonumber
x_{12b} x_{3ac}x_{(13b,2\bar b \bar c)} x_{(12 \bar a,3 \bar b \bar c)} x_{(13 \bar a,2ac)} 
-x_{13b} x_{2ac} x_{(12b,3\bar b \bar c)}  x_{(13 \bar a,2 \bar b \bar c)}x_{(12 \bar a,3ac)} \\ \nonumber
x_{(12a,13b)}x_{(13a,12c)}x_{(12b,13c)} -x_{(13a,12b)}x_{(12a,13c)}x_{(13b,12c)} \\ \nonumber
x_{(12a,13b)}x_{(13a,12c)}x_{(12b,23c)} x_{(23b,13c)} 
-x_{(13a,12b)}x_{(12a,13c)}x_{(23b,12c)} x_{(13b,23c)}  \nonumber
\end{eqnarray}
\\[-2em]
for suitable choices of $a,b,c,  b',c', \bar a ,b< \bar b<\bar c$.

This set of binomials consists of certain square-free, multi-homogenous binomials in $\sR$, 
linear in $\vr$-variables of $\PP_F$, for every $F$ whose $\vr$-variables appear in them.
\end{itemize}
\end{thm}
\begin{proof}  The birational map $\Theta$ corresponds to the following homomorphism
$$\vi_\Gr: R_0 \otimes \bigotimes_{F \in \sF} \kk[x_{(\uu_s,\uv_s)}]_{s \in S_F}
 \lra R_0/\langle F \mid F \in \sF \rangle$$
$$ \vi_\Gr|_{R_0}=\Id_{R_0}, \;  x_{(\uu_s,\uv_s)} \to x_{\uu_s}x_{\uv_s}$$
where $R_0=\kk[x_{ijk}]_{ijk \ne 123}$, the affine algebra of $\bU$,
 and $\kk[x_{(\uu_s,\uv_s)}]_{s \in S_F}$ is graded. Then, $\sV \subset \sR$ is defined by
the multi-homogeneous kernel $\ker^\mh (\vi_\Gr)$. 

One checks directly that relations in (GL), (GB), and (NGB1) belong to $\ker^\mh (\vi_\Gr)$.
The homomorphism  $\vi_\Gr$ is explicitly simple.
The fact that (GL), (GB), (NGB1), and (NGB2) together generate  $\ker^\mh (\vi_\Gr)$
is proved in \S 4 of \cite{Hu2025}.
\end{proof}
 
We let $L_\sF=\{L_F \mid F \in \sF\}$, $\cB^\gov$ be the set of all governing binomials,
 $\cB^\ngvr$ the set of all binomials in (NGB1), and
 $\cB^\ngvh$ the set of all binomials in  (NGB2).
We set $\cB^\ngv=\cB^\ngvr \sqcup \cB^\ngvh$.

The scheme $\sR$ is a product of affine space 
with projective spaces, hence can be covered by
 affine charts $\{\fV\}$ with a system $\var_\fV$ of affine coordinates
consisting of $\vp$-variables and de-homogenized $\vr$-variables. We let
$\var_\bU =\{x_{ijk} \mid  ijk \ne 123\}.$

\begin{thm}\label{universal-eqs-Gamma} {\rm (Lemma 7.3, \cite{Hu2025})}
Given  any subset  $\Gamma \subset \var_\bU$
such that $Z_\Gamma$ is integral, then the following hold
\begin{itemize}
\item there exists a closed subscheme $Z_{\sF,\Ga}$ of $\sV$
with an induced morphism 
 $$Z_{\sF,\Ga} \to Z_\Ga;$$
\item $Z_{\sF,\Ga}$ comes equipped with an irreducible component  
$Z^\dagger_{\sF,\Ga}$ with the induced surjective morphism 
$Z^\dagger_{\sF,\Ga}  \to Z_\Ga$ that is proper and birational;
 \item  for any affine chart $\fV$ of $\sR$ such that
$Z_{\sF,\Ga} \cap \fV \ne \emptyset$, there  exists a subset, 
$$ \tGa^\zero_{\fV} \; \subset \;  \var_\fV$$
such that $Z_{\sF,\Ga} \cap \fV$ as a closed subset of $\fV$ is defined by
\begin{eqnarray}\label{uni-eqs} L_F|_\fV,  \; \forall \; F \in \sF; \; B|_\fV, \; \forall \; B \in \cB^\gov \sqcup \cB^\ngv \\
 y, \;\forall \;  y \in  \tGa^\zero_{\fV}. \nonumber
\end{eqnarray}
\end{itemize}
\end{thm}
Observe that the equations in \eqref{uni-eqs} are uniform
 for all $Z_\Gamma$, thus we refer it
as {\it universal local equations} for (birational transforms) of all singularities.\footnote{
Hence, to resolve them, there is no reason to focus on any particular one nor to analyse their singularities. Thereby our blowup process is universal
and simultaneously resolves all possible singularities.}

\subsection{Behind the universal blowups: Jacobian, Jacobian, Jacobian}\label{foc-jac} $\ $

Let us try to investigate  the Jacobian matrix of the relations in Theorem \ref{universal-eqs}.

Indeed, let us try to compute the Jacobian matrix of the governing relations in
 $$L_\sF \sqcup \cB^\gov.$$

We foresee a maximal minor of $\Jac(L_\sF \sqcup \cB^\gov)$ that is
lower-triangular, block-wise. This requires a total order of the set $\sF$.

\begin{defn}\label{order-sF} 
Consider the set 
$$\sF=\{F_{1uv}, F_{2uv}, F_{3uv}, F_{abc} \mid 3<u < v \le n, \; 3<a<b<c \le n \}.$$
We define $$F_{1uv}< F_{2uv}< F_{3uv} <F_{abc}$$
for all $ 3<u < v \le n,  3<a<b<c \le n$. Given two $uv \ne u'v'$, we define
 $$F_{iuv}< F_{ju'v'} \; \iff \; uv <_\lex u'v', \; \forall \; i, j \in \{1,2,3\}.$$
Given two $abc \ne a'b'c'$, we define
 $$F_{abc}< F_{a'b'c'} \; \iff \; abc<_\lex a'b'c'.$$
\end{defn}
%This provides $\sF$ a well-defined total order which we use throughout.

For any $F \in \sF$, we let $\cB^\gov_F$ be  consist of all the governing binomials associated with $F$
and  $\fG_F=\{L_F\} \sqcup \cB^\gov_F$. Then we obtain  the set of  all governing relations
$$\fG=\{\fG_F \mid F \in \sF\}$$ which  inherits the total order from that of $\sF$.

%With its order, we have the following useful observation for computing $\Jac(\fG)$.

\begin{prop-defn}\label{pleas} For any $F \in \sF$, the leading variable
$x_{\uu_F}$ (Definition \ref{led-via})
of $F$ and the $\vr$-variables of $\cB^\gov_F$ do not appear in any relations of $\fG_{F'}$ if $F' <F$.
Such a variable is called pleasant.
\end{prop-defn}

It leads to the following triangular maximal minor of $\Jac(\fG)$
\begin{equation}\label{tri-jac} %\footnotesize{
{{\partial (\fG)} \over {\partial (\hbox{leading, $\vr$-variables})}} =\left(
\begin{array}{cccccccccc}
{{\partial (\fG_{F_1})} \over {\partial (\hbox{$x_{\uu_{F_1}}, \vr$-vars})}}  & 0 & \cdots & 0 \\
* & {{\partial (\fG_{F_2})} \over {\partial (\hbox{$x_{\uu_{F_2}}, \vr$-vars})}} & \cdots & 0 \\
\vdots & \vdots & \cdots & \vdots \\
* & * & \cdots & {{\partial (\fG_{F_\up})} \over {\partial (\hbox{$x_{\uu_{F_\up}}, \vr$-vars})}} 
\end{array}
\right) %}
\end{equation} where $\up$ is the cardinality of $\sF$.

On the open subset $\cV^0 \subset \cV$
 consisting of points at which all $\vp$- and $\vr$-variables
do not vanish, using   \eqref{tri-jac}, one computes and concludes that for any $\bz \in \cV^0$,
\begin{equation}\label{key-dim-cal}
 \dim_\bz \cV^0 = \dim_\bz \sR - \rk {\rm Jac} (\fG)(\bz),
\end{equation}
implying that over $\cV^0$, relations in $\cB^\ngv$ 
are dependent.
%\footnote{Thus, the governing relations only define
%a reducible closed scheme. The roles of the non-governing relations
%in $\cB^\ngv$ are to  pin down its main component $\sV$. 
%Our  universal blowup process (explained below) will eventually blow 
%all the boundary components  out of existence, making  
%the proper transforms of all the non-governing relations
%in $\cB^\ngv$ redundant in the end.}
Thereby we call relations
in $$L_\sF \sqcup \cB^\gov$$ the governing relations.\footnote{
 \eqref{tri-jac} and \eqref{key-dim-cal} motivate us to simply
 focus  on  governing relations and to make ${\rm Jac} (\fG)$ full rank everywhere.
\emph{{\it The goal is clear;  the plan is simple.}} }

Before moving on, observe that the total order on $\sF$ (Definition \ref{order-sF})
induces a total order on the set $\cB^\gov=\bigsqcup_F \cB^\gov_F$.

\subsection{Proper transforms of binomials of $\cB^\gov$:  Designing  the universal blowups} $\ $

We inductively construct our universal blowups, block by block, according to the order
$$\fG_{F_1}< \fG_{F_2}< \cdots < \fG_{F_\up}.$$

Any universal blowup $$\tsR_* \lra \tsR_{*-1}$$ is motivated locally,  computed locally, but
 defined globally.  Hence, we introduce
%for global definition, we need the following.

\begin{defn} For any $\uu \in \{ijk \mid ijk \ne 123\}$, we let 
$$X_\uu=(x_\uu =0) \subset \sR, $$
called a $\vp$-divisor. For any coordinate $x_{(\uu,\uv)}$ of $\PP_F$ for some $F \in \sF$,
we let 
$$X_{(\uu,\uv)}=(x_{(\uu,\uv)}=0) \subset \sR,$$
called $\vr$-divisor. For any $F \in \sF$, we let 
$$D_F=(L_F =0) \subset \sR,$$  called a $\fL$-divisor. 
 We let $\cD_\vp$ be the set of all $\vp$-divisors, 
$\cD_\vr$ the set of all $\vr$-divisors, and
$\cD_\sF$  the set of all $\fL$-divisor.
\end{defn}

Over $\sR$, we have  the set $\cD_\sR$ of smooth divisors
$$\cD_\sR=\cD_\vp \sqcup \cD_\vr \sqcup \cD_\sF.$$ 
%For any subset $\sF_1 \subset \sF$, we let 
%$\cL_{F_1}$ be the pull-back to $\sR$ of the line bundle
%$\bigotimes_{F \in \sF_1} \cO_{\PP_F}(1)$ and 
%$$\cL_\sR =\{ \cL_{F_1} \mid \cF_1 \subset \cF\}.$$
%Then, every divisor $D$ of the set $\cD_\sR=\cD_\vp \sqcup \cD_\vr \sqcup \cD_\sF$ corresponds
 %to section $s_D$ in the form of $x_\uu, x_{(\uu,\uv)}, L_F$ of a line bundle in $\cL_\sR$.
%$\cL=\cO_\bU \otimes \bigotimes_{F \in \sF} \cO_{\PP_F}(1).$

For any universal blowup $\tsR_* \lra \tsR_{*-1}$, inductively we let
$$\cD_{\sR_*} = \widehat\cD_{\sR_{*-1}} \sqcup E_*$$
where $\widehat \cD_{\sR_{*-1}}$ consists of proper transforms of divisors of $ \cD_{\sR_{*-1}} $
and $E_*$ is the exceptional divisor of the blowup $\tsR_* \lra \tsR_{*-1}$.

For any $B \in \cB^\gov$, we can express $B=T_B^+ - T_B^-$ as
$$T_B^+ = x_{\uu_F}x_{(\uu_s,\uv_s)}, \;  T_B^-=x_{\uu_s}x_{\uv_s}x_{(123,\uu_F)}, \;\;
\hbox{for some $s \in S_F$.}$$

The  center of a $\wp$-blowup $\sR_{*} \lra \sR_{*-1}$ is constructed from
 $T_{B_{*-1}}^\pm$, the proper transform of 
$T_B^\pm$ for some  $B\in \cB^\gov$.
Locally over a chart $\fV'$ of $\sR_{*-1}$, 
we pick a factor $y_\pm$ in  the proper transform $T_{\fV', B}^\pm$ of $T_B^\pm$ in $\fV'$,
and blow up along
$$(y^+=0)\cap (y^-=0).$$
Each of $(y^\pm=0)$  corresponds to  global divisor $D_\pm \in \cD_{\sR_{*-1}}$
and we let
$$\tsR_{*-1} \lra \sR_*$$ be the blowup of $\tsR_{*-1}$ along
the codimension two locus $D_+ \cap D_-$.

After the earlier blowups,    $T_{\fV', B}^\pm$ might have been changed.
By Proposition-Definition \ref{pleas}, 
the leading variable $x_{\uu_F}$ and the $\vr$-variable $x_{(\uu_s,\uv_s)}$ 
are pleasant, implying that before arriving at the block $\fG_F$, over any chart $\fV''$, 
the proper transform $$T_{\fV'',B}^+= x_{\fV'', \uu_F}x_{\fV'', (\uu_s,\uv_s)}$$ remains unchanged,
{\it limiting the choices of $y^+$ to  certain few in $T_{\fV',B}^+$}.
But, the variable $y^-$ may appear in $T_{\fV',B}^-$ with multiplicity $m>0$.

\begin{lemma}\label{higher-mul-dec} {\rm (\cite{Hu2025})} Over any chart $\fV$ of $\tsR_*$ 
lying  over $\fV'$ of $\sR_{*-1}$, the proper transform
$T_{\fV,B}^+$ always remains square-free.
%\footnote{It is crucial to our approach that
%we always have quite explicit firm control on $T_{\fV,B}^+$.}. 
And
\begin{itemize}
%\item $T_{\fV,B}^+$ always remains linear. And 
\item either $\deg T_{\fV, B}^+ < \deg T_{\fV', B}^+$,  or 
\item $(y_-)^m$ in $T_{\fV', B}^-$  is replaced by $(\zeta_{E_*})^{m-1}$ in  $T_{\fV, B}^-$
where $\zeta_{E_*}$ is a coordinate of $\fV$ such that $ E_* \cap \fV=(\zeta_{E_*}=0)$.
\end{itemize}
\end{lemma}

%The above explains the ideas and outline the process.
%In what follows, we formalize the procedure.

\section{Universal blowups and equations}

\subsection{Universal $\vt$-blowups} $\ $

\begin{comment}
Our aim is, for any governing binomial $B$, 
to eliminate common zero locus of a section of the monomial $T_B^+$ and
a section of the monomial $T_B^-$. Fix a block $\fG_F$ for some $F \in \sF$, one finds
that the pair of sections $(x_{\uu_F} \in T_B^+, x_{(123, \uu_F)} \in T_B^-)$ hold for
all $B \in \fG_F$. Thus, at least for the purpose of efficiency,  it motivates us to
begin with orderly blowing up the loci $(x_{\uu_F}=0)\cap (x_{(123, \uu_F)}=0)$, 
and,  this is the sequential $\vt$-blowups.

Consider a line bundle $L=\bigotimes_i L_i$ with a canonical inclusion $\cO \subset L_i$, 
a section $s \in \Ga( L)$ is a monomial section
if $s=\bigotimes_i s_i$, often written as $s=\prod_i s_i$, for  some $s_i \in \Ga(L_i)$ 
where some of these $s_i$ are allowed to be
the trivial section of $\cO \subset L_i$. For example, for any binomial $B$ of
$\cB^\gov \sqcup \cB^\ngv$, write $B=T_B^+=T_B^-$, then $T_B^\pm$ is a monomial section
of $\cL_\sR=\cO_\bU \otimes \bigotimes_{F \in \sF} \cO_{\PP_F}(1).$
\end{comment}

Fix any  $F \in \sF$ and consider all $B \in \cB^\gov_F$. We have
$$B = x_{\uu_F} x_{(\uu_s, \uv_s)} - x_{\uu_s}x_{\uv_s} x_{(123,\uu_F)}, \;
s \in S_F.$$
Observe that the pair $(x_{\uu_F}, x_{(123,\uu_F)})$ appears in every member of $\cB^\gov_F$.
Thereby we  let 
$$Z_{\vt_F}  = X_{\uu_F} \cap X_{(123,\uu_F)}.$$
Then we have
$$Z_{\vt_{F_1}} < Z_{\vt_{F_2}} < \cdots < Z_{\vt_{F_\up}},$$
and we blow up $\sR$ along these centers, in that order,  to obtain
$$\tsR_\vt:=\tsR_{\vt_{F_\up}} \lra \cdots \lra \tsR_{\vt_{F_1}} \lra \sR.$$
%where inductively $\tsR_{\vt_h} \lra \tsR_{\vt_{h-1}}$
% is the blowup of $\tsR_{\vt_{h-1}}$ along the proper
%transform of $Z_{\vt_{F_h}}$ and $\tsR_{\vt_{0}} :=\sR$.
It induces
$$\tsV_\vt:=\tsV_{\vt_{F_\up}} \lra \cdots \lra \tsV_{\vt_{F_1}} \lra \sV$$
where  $\tsV_{\vt_{F_h}}$ is the proper transform of $\tsV$. 

Each blowup $\tsR_{\vt_{F_h}} \lra \tsR_{\vt_{F_{h-1}}}$, where $\tsR_{\vt_{F_0}} =\sR$,
 creates an exceptional divisor $E_{\vt_{F_h}}$ of  $\tsR_{\vt_{F_h}}$,
and we let $E_{\vt, F_h}$ be the proper transform of  $E_{\vt_{F_h}}$ in $\tsR_{\vt}$.
%If $F=F_h$, we also write $E_{\vt, F}$ for $E_{\vt, h}$.
The set of all  these exceptional divisors is denoted $\cE_\vt$, together with proper transforms of
divisors of $\cD_\vp$ and $\cD_\vr$, we obtain\footnote{Throughout \cite{Hu2025},
we  laboriously introduce  a set of explicit divisors for every universal blowup,  because
 global centers  are defined in terms of them.}
$$\cD_\vt= \cD_{\vt, \vp} \sqcup \cD_{\vt, \vr} \sqcup \cE_\vt.
$$

%The scheme $\tsR_\vt$ is smooth and can be covered by affine charts $\{\fV\}$. 
For any intermediate universal blowup $\tsR_*$ (e.g. $\tsR_\vt$),
 we write the proper transforms of $T_B^\pm$, $L_F$, 
 $x_\uu$, $x_{(\uu,\uv)}$, $X_{(\uu,\uv)}$, $E_{\vt_F}$, etc. as
  $$T_{*, B}^\pm, \; L_{*,F}, \; x_{*, \uu}, \; x_{*, (\uu,\uv)}, \; X_{*,(\uu,\uv)}, \;E_{*, F},  \; \hbox{ ect.}$$ 
Similarly, for any chart $\fV$ of  $\tsR_*$,
 we write their proper transforms over $\fV$ as
  $$T_{\fV, B}^\pm, \; L_{\fV,F}, \; x_{\fV, \uu}, \; x_{\fV, (\uu,\uv)},  \; 
X_{\fV,(\uu,\uv)}, \; E_{\fV, F},  \; \hbox{ ect.}$$

\begin{thm}\label{vt-blowup}  {\rm (Chapter 6, \cite{Hu2025})} 
We have $$\tsV_\vt \cap X_{\vt, (123, \uu_F)} = \emptyset, \;\;
\hbox{ for all $F \in \sF$. } $$ Consequently, $\tsV_\vt$  can be covered by %the so-called preferred 
affine charts $\{\fV\}$ such that
\begin{enumerate}
\item For  any  binomial $B$ in $\cB^\gov$,
$T_{\fV, B}^+ =x_{\fV, \uu_F} x_{\fV, (\uu_s, \uv_s)}$,
keeping its original form.
\item The proper transforms of all relations in $\cB^\ngvr$
%Theorem \ref{universal-eqs} (NGB1) 
become dependent and can be discarded from consideration.
\item  The proper transforms of all relations in $\cB^\ngvh$
 remain to be square-free over any chart.
\item For any $F \in \sF$, when $E_{\vt, F}\cap \fV =\emptyset$,  
$$(\alpha) \;\;\;\; 
L_{\fV, F} = x_{\fV, (123,\uu_F)} + \sum_{s \in S_F \- \{s_F\}} \sgn(s) x_{\fV, (\uu_s,\uv_s)};$$
when $E_{\vt, F}\cap \fV \ne \emptyset$,
$$(\beta) \;\;\;\; 
L_{\fV, F} = \de_{\fV,(123,\uu_F)} + \sum_{s \in S_F \- \{s_F\}} \sgn(s) x_{\fV, (\uu_s,\uv_s)}$$
where $\de_{(123,\uu_F)}$  is a coordinate of $\fV$ such that
$E_{\vt, F}\cap \fV =(\de_{(123,\uu_F)}=0)$.
\end{enumerate}
\end{thm}

\subsection{Universal $\wp$- and $\ell$-blowups} $\ $

\begin{defn}\label{termi}
Let $\tsR_*$ be any (intermediate) blouwp scheme covered by charts $\{\fV\}$. We say
$B$ terminates on $\fV$ if $T_{\fV, B}^\pm$ is invertible along $\fV \cap \tsV_*$. 
We say $B$ terminates on $\tsR_*$ if it terminates on all charts.
\end{defn}

The goal is to make every $B \in \cB^\gov$ terminate.

We consider all the governing  binomial $B$ associated to  $F_k \in \sF$ for some fixed $k \in [\up]$.
We denote this process by $(\wp_k)$.
%Inductively, after all the previous blowups, 
%while $T_{B_*}^+$ of  the proper transfrom$B_*$ always remains multiplicity-free,
%the minus term $T_{B+*}^-$ of $B_*$ might acquire variables with
%higher multiplicities. 
Locally over a chart $\fV$, we select all possible $y^\pm$ from  $T_{\fV,B}^\pm$, for all $B$ associated to  $F_k$, 
with the corresponding global divisors $D^\pm$ and let
$$\cZ_{\wp_k \fr_\mu} = \{D^+ \cap D^-\}.$$
Each $D^+ \cap D^-$ is called a $\wp_k$-center.
We provide a total order on $\cZ_{\wp_k \fr_\mu}$.\footnote{This can be subtle, 
see, the proof of Theorem \ref{wp-ell-blowup} (1).}
We then blow up along the centers in $\cZ_{\wp_k \fr_\mu}$, in that order.
This completes one round of $\wp$-blowups, which we denote $(\fr_\mu)$ for some 
$\mu \in \NN$.
If the center $D^+ \cap D^-$ is the $h$-th member in $\cZ_{\wp_k \fr_\mu}$,
we then denote the $\wp$-blowup along $D^+ \cap D^-$  as
$$(\wp_k\fr_\mu\fs_h),$$
which reads: phase $k$, round $\mu$, and step $h$.
By Lemma \ref{higher-mul-dec}, a mutiple rounds of $\wp$-blowups may be required to
 bring down all  multiplicities of the term $T_{B_*}^-$ to one.
% after which by the lemma below,  the binomial $B$ terminates.

\begin{defn}  We say $B_\fV$ is $\wp$-multiplicity-free if for
any variable $y$ in $T_{\fV,B}^\pm$, it is either invertible
or has multiplicity 1. We say $B$ is $\wp$-multiplicity-free if
it is $\wp$-multiplicity-free for all charts.

We say a center $D^+ \cap D^-$ is tamed if over any chart
$\fV$ with $(y^+=y^-0) = D^+ \cap D^- \cap \fV$,
whenever $y^\pm \mid T_{\fV, B}^\pm$ , we have
$y^\mp \mid T_{\fV, B}^\mp$, for all $B$ associated with $F_k$.
\end{defn}

\begin{lemma}\label{mul=1=1roundterm} 
Suppose in $(\wp_k\fr_\mu)$, 
all $B$ associated with $F_k$ are $\wp$-multiplicity-free,
and all $\wp_k$-centers are tamed,
then all such $B$ terminate after just one round
 of $\wp$-blowups.
\end{lemma}
\begin{proof} This follows by repeatedly applying Lemma \ref{higher-mul-dec}.
\end{proof}

\begin{cor} Every $B \in \cB^\gov$ terminates
after a finite sequential $\wp$-blowups.
\end{cor}

\begin{comment}
Given any divisor $D$ of $\tsR_(\wp_k\fr_\mu\si_{h-1}))$,
we let $D_{(\wp_k\fr_\mu\si_h)}$ be its proper transform in 
$\tsR_(\wp_k\fr_\mu\si_{h}))$, for example, $X_{(\wp_k\fr_\mu\si_h), \uu}$ is
the proper transform of the $\vp$-divisor $X_\uu$. The same applies to all other symbols.
Thus, $\cD_{(\wp_k\fr_\mu\si_h)}$ consists of the proper transforms of 
$\cD_{(\wp_k\fr_\mu\si_{h-1})}$ together with the exceptional divisor 
$\cD_{(\wp_k\fr_\mu\si_{h})}$ of the blowup
$$f_{(\wp_k\fr_\mu\si_{h})}: \tsR_{(\wp_k\fr_\mu\si_{h})} \lra \tsR_{(\wp_k\fr_\mu\si_{h-1})}.$$
\end{comment}

Consider any universal blowup $\sR_* \lra \sR_{*-1}$. 
Locally over a chart $\fV'$ of $\sR_{*-1}$, the blowup is always  along
a co-dimensional two center
$(y^+=0) \cap (y^-=0)$. % where $y^\pm \mid T_{\fV',B}^\pm$ for 
%some$B \in \cB^\gov$. 
For any chart $\fV$ of $\sR_*$,
 $\fV \subset \fV' \times \PP^1_{[\xi,\eta]}$ corresponds to
 $(\xi \ne 0)$ or $(\eta \ne 0)$. 

\begin{lemma}\label{term-charts} Suppose $B_\fV$ terminates, then
$\tsV_* \cap \fV$ is covered by $(\xi \ne 0) \cap (\eta \ne 0)$.
\end{lemma}
\begin{proof} This is clear.
\end{proof}

\begin{comment}
\begin{cor}\label{Tplus=known} 
When the proper transform of $B_*$ terminates on $\tsR_*$, 
we can choose charts $\{\fV\}$ to cover $\tsV_*$ such that $T_{*,B}^+$ takes 
the following two forms
$$T_{\fV, B}^+= y_{\fV, \uu_F} x_{\fV, (\uu_s,\uv_s)}, \; T_{\fV,B}^+= x_{\fV, (\uu_s,\uv_s)}$$
where $y_{\fV,\uu_F}$ is a transform of  $x_{\uu_F}$. Further, 
 $y_{\fV, \uu_F}$ and $x_{\fV, (\uu_s, \uv_s)}$  are pleasant variables.
\end{cor}
\end{comment}

%The $\wll$-blowups will be explain the proof the following theorem.

We let $\tsR_\ell$ be the final blowup scheme and $\tsV_\ell$ the proper transform of $\sV$,
where $\ell$-blowups will be seamlessly introduced in the proof of the theorem below.

\begin{thm}\label{wp-ell-blowup} {\rm (Chapter 6, \cite{Hu2025})}  The final blowup scheme
$\tsV_\ell$  can be covered by  smooth  charts $\{\fV\}$ of $\tsR_\ell$
such that the following holds. 
\begin{enumerate}
\item For  any governing binomial $B$ in $\cB^\gov$, we have either
$$T_{\fV, B}^+=  a x_{\fV, (\uu_s,\uv_s)}
 \;\; \hbox{or} \;\; T_{\fV,B}^+= a y_{\fV, \uu_F} x_{\fV, (\uu_s,\uv_s)}$$
where $a$ is invertible and $y_{\fV,\uu_F}$ is a transform of  $x_{\uu_F}$. Importantly, 
 $y_{\fV, \uu_F}$ and $x_{\fV, (\uu_s, \uv_s)}$  are
 both pleasant variables.\footnote{Hence we have
complete control of $T_B^+$, though
  little for the minus term $T_B^-$.
Thereby we can compute  $\Jac(\fG)$  because we
know explicitly the final forms of $T_B^+$.}
\item For any $F \in \sF$,  $L_{\fV, F}$ takes one of the following two forms:
\begin{equation}\label{case-alpha}\nonumber
(\alpha) \;\; 
L_{\fV, F} = \sgn(s_F) x_{\fV, (123,\uu_F)}+ \sum_{i=1}^l \sgn(s_i) x_{\fV, (\uu_{s_i},\uv_{s_i})} + 
 \sum_{j=1}^q \sgn(s_j) x^*_{\fV, (\uu_{s_j},\uv_{s_j})}
\end{equation}
where $x^*_{\fV, (\uu_{s_j},\uv_{s_j})}$ denotes the pull-back; 
\begin{equation}\label{case-beta}\nonumber
(\beta) \;\;\;\;\;\; 
L_{\fV, F} = \sgn(s_F)  y_{\fV, (123,\uu_F)} +1
\end{equation}
where $y_{\fV, (123,\uu_F)}$ is the proper
transform of the local parameter $\de_{(123,\uu_F)}$ for $E_{\vt,F}$.
\end{enumerate}
\end{thm}
\begin{proof}
(1). Consider any $B \in \cB^\gov_F$ for some $F \in \sF$. Write
$$B= x_{\uu_F} x_{(\uu_s,\uv_s)} - x_{\uu_s} x_{\uv_s} x_{(123, \uu_F)}.$$
When the blowups are performed for the blocks $\fG_{F'}$ with $F' < F$, the
plus term $$T_B^+= x_{\uu_F} x_{(\uu_s,\uv_s)}$$  is not affected and remains the same form.
When the block $\fG_F$ is treated, we order the centers of the $\wp$-blowups
 by declaring  the $\vr$-variable
$x_{(\uu_s,\uv_s)}$ being the largest, and $x_{\uu_F}$ the second largest. That is to say,
 the centers having the $\vr$-variable
$x_{(\uu_s,\uv_s)}$ are blown-up last. This way, when the binomial $B$ terminates, we retain the
 $\vr$-variable $x_{(\uu_s,\uv_s)}$ in $T_B^+$, due to Lemma \ref{term-charts}.
The leading variable $x_{\uu_F}$, being the second largest, its transform may or may not remain, 
 depending on whether the $\vr$-variable $x_{(\uu_s,\uv_s)}$ is invertible,
so we have
the two cases as stated.\footnote{Pleasantly, these are the pleasant variables that we use to compute the final Jacobian, \S \ref{jjj}.}
 
(2)  When the governing binomials of $\fG_F$ all terminate, 
we obtain the final $\wp$-blowup  scheme for that block
$$\tsV_{\wp_F} \subset \tsR_{\wp_F},$$
equipped with the proper transform 
$E_{\wp_F, F}$ of $E_{\vt, F}$ and the proper transform $D_{\wp_F, F}$ of $D_F=(L_F=0)$.
We  move on to consider $L_F$. 
We cover  $\tsV_{\wp_F}$ by smooth charts $\{\fV'\}$ of $\tsR_{\wp_F}$.
 When  $E_{\wp_F, F} \cap \fV'=\emptyset$, by Theorem \ref{vt-blowup} (4) $(\alpha)$,
%over any  chart $\fV$ lying over $\fV'$,  
$L_F$ takes of the form
$$L_{\fV', F} = \sgn(s_F) x_{\fV', (123,\uu_F)} + \sum_{s \in S_F \- \{s_F\}} \sgn(s) x^*_{\fV', (\uu_s,\uv_s)}.$$
%as we always have $X_{\vt, (123,\uu_F)} \cap \tsV_\vt =\emptyset$.
When $E_{\wp_F, F} \cap \fV' \ne \emptyset$, by Theorem \ref{vt-blowup} (4) ($\beta$),
$$L_{\fV', F} = \sgn(s_F) \de_{\fV',(123,\uu_F)} + 
\sum_{s \in S_F \- \{s_F\}} \sgn(s) x^*_{\fV', (\uu_s,\uv_s)}$$
where $\de_{\fV', (123,\uu_F)}$  is a coordinate of $\fV'$ such that
$E_{\wp_F, F}\cap \fV' =(\de_{\fV', (123,\uu_F)}=0)$.

Now, we blow up $ \tsR_{\wp_F}$ along
$E_{\wp_F, F} \cap D_{\wp_F, F}$ to obtain the $\ell$-blowup: 
$$ \tsR_{\ell_F} \lra \tsR_{\wp_F}.$$ 

Consider any chart $\fV$ of $ \tsR_{\ell_F} $ lying over $\fV'$.

The first case in the above discussion 
implies (2) ($\alpha$),  where we let $x_{\fV', (\uu_{s_i},\uv_{s_i})}, 1 \le i \le l$,
be all the $\vr$-variables that are invertible on the blowup center. 

In the second case, 
by applying (the same argument of) Lemma \ref{term-charts},
(2) ($\beta$) follows.
\end{proof}

\section{Resolution of singularities}

\subsection{Birational transforms of singularities and equations} $\ $

For any intermediate blowup $\pi: \tsR_* \lra \tsR_{*-1}$ and a chart $\fV'$ of $\tsR_{*-1}$, 
$\pi^{-1}(\fV') \lra \fV'$ is the blowup along a locus $(y_0'=0) \cap (y_1'=0)$.

\begin{lemma}\label{for-slice}\footnote{This lemma guarantees 
 the existence of birational transform of $Z_\Ga$.}
For any chart $\fV$ of $\tsR_* $ lying over $\fV'$, 
we can assume  $y_i'$ turns into the exceptional variable
$\zeta$ for some $i \in \{0, 1\}$ and
$y_j$ is the proper transform of  $y_j'$ with $j=\{0,1\}\-\{i\}$.
Then for any binomial $B \in \cB^\gov \sqcup  \cB^\ngv$
and any term $T_{\fV,B}$ of $B_\fV$,
if $y_j \mid T_{\fV, B}$, then either $T_{\fV, B}$ is linear in $y_j$,
 or else, $\zeta \mid T_{\fV, B}$.
\end{lemma}
\begin{proof}
For $\vt$- and $\wp$-blowups,  the 
lemma holds  because one of the variables of $y_0'$ and $y_1'$ arises
from the plus term $T_B^+$ of some $B \in \cB^\gov$ and we have complete control of any proper transform of $T_B^+$. Similar reasoning applies to $\ell$-blowups.
\end{proof}

\begin{thm}\label{final-eqs-Gamma} {\rm (Corollary 7.6, \cite{Hu2025})}
Given  any subset  $\Gamma \subset \var_\bU$
such that $Z_\Gamma$ is integral, then the following hold
\begin{itemize}
\item there exists a closed subscheme $\tZ_{\ell,\Ga}$ of $\tsV_\ell$
with an induced morphism 
 $$\tZ_{\ell,\Ga} \to Z_\Ga;$$
\item $\tZ_{\ell,\Ga}$ comes equipped with an irreducible component  
$\tZ^\dagger_{\ell,\Ga}$ with the induced surjective morphism 
$\tZ^\dagger_{\ell, \Ga}  \to Z_\Ga$ that is proper and birational;
 \item  for any affine chart $\fV$ of $\sR_\ell$ as in Theorem \ref{wp-ell-blowup}
such that
$\tZ_{\ell,\Ga} \cap \fV \ne \emptyset$, there  exist subsets, 
$$ \tGa^\zero_{\fV}, \;  \tGa^\one_{\fV} \; \subset \;  \var_\fV$$
such that $\tZ_{\ell,\Ga} \cap \fV$ as a closed subset of $\fV$ is defined by
\begin{eqnarray}\label{uni-eqs-ell} L_{\fV, F},  \; \forall \; F \in \sF; \; B_\fV, \; \forall \; B \in \cB^\gov \sqcup \cB^\ngv \\
 y, \;\forall \;  y \in  \tGa^\zero_{\fV}, \; y -1, \;\forall \;  y \in  \tGa^\one_{\fV}. \nonumber
\end{eqnarray}
\end{itemize}
\end{thm}
\begin{proof}
The closed subscheme $\tZ_{\ell,\Ga}$  is constructed inductively.
When an intermediate $\tZ_{{*-1},\Ga}$ is not  contained in the blowup center
of $\tsR_* \lra \tsR_{*-1}$, we simply let $\tZ_{{*},\Ga}$ be the proper transform of
$\tZ_{{*-1},\Ga}$, thereby the subset $\tGa^\zero_{\fV}  \subset  \var_\fV$.
 However, when $\tZ_{{*-1},\Ga}$ is   contained in the blowup center,
we need to construct a birational slice of $\tZ_{{*-1},\Ga}$ in the exceptional divisor,
such a slice exists due to Lemma \ref{for-slice}. Lemma \ref{for-slice} implies that
locally over the chart $\fV$, when setting the exceptional variable $\zeta$
to be zero (since $\pi^{-1}(\tZ_{{*-1},\Ga})$  is contained in the exceptional divisor),
we obtain linear equations in $y_j$, or no constrains at all. In the formal case, we solve
the linear equations generally and take closure; in the latter case, we take the trivial slice
through the point $[1,-1] \in \PP^1_{[\xi,\eta]}$, thereby the subset 
$\tGa^\one_{\fV}  \subset  \var_\fV$.
\end{proof}

\subsection{Jacobian, Jacobian, Jacobian}\label{jjj}

\begin{thm}\label{tZga=smooth} {\rm (Theorem 8.5, \cite{Hu2025})}
Consider any subset  $\Gamma \subset \var_\bU$
such that $Z_\Gamma$ is integral. Then, $\tZ_{\ell,\Gamma}$ is smooth. In particular,
$\tZ^\dagger_{\ell,\Gamma}$ is smooth, 
\end{thm}
\begin{proof}
We apply Theorems \ref{wp-ell-blowup} and \ref{final-eqs-Gamma},
and follow their notation.

Fix and consider a block $\fG_F$.

$\bullet$ Case $(\alpha)$ as in Theorem \ref{wp-ell-blowup} (2).

In this case, the chart $\fV$ lies over $(x_{(123,\uu_F)} \ne 0)$.

Following notation of Theorem \ref{final-eqs-Gamma},
we let $\tGa_\fV= \tGa^\zero_{\fV} \sqcup \tGa^\one_{\fV}$.
Using notation of Theorem \ref{wp-ell-blowup} (2) ($\alpha$),
we  introduce the following maximal minor of  $\Jac (\fG_{\fV,F}|_{\tGa_\fV})$
  %$J(\cB^\gov_{\fV,F}|_{\tGa_\fV}, L_{\fV,F}|_{\tGa_\fV})$
%$$J^*(\cB^\gov_{\fV,F}|_{\tGa_\fV}, L_{\fV,F}|_{\tGa_\fV})
$$J^*(\fG_{\fV,F}|_{\tGa_\fV})= {{\partial(B_{\fV, s_1}|_{\tGa_\fV} \cdots B_{\fV, s_{\l}}|_{\tGa_\fV}, B_{\fV, t_1}|_{\tGa_\fV} \cdots B_{\fV, t_q}|_{\tGa_\fV},
L_{\fV,F}|_{\tGa_\fV})} \over {{\partial(y_{\uu_F},
x_{(\uu_{s_1},\uv_{s_1})} \cdots x_{(\uu_{s_{\l}},\uv_{s_{\l}})},
x_{(\uu_{t_1},\uv_{t_1})} \cdots x_{(\uu_{t_q},\uv_{t_q})}
 )}}} .$$
Then, one computes and finds that at a point $\bz$, it is equal to
\begin{eqnarray} \nonumber
% J^*(\cB^\gov_{\fV,F}, L_{\fV,F}|_{\tGa_\fV}) =   
% x_{\uu_F}, x_{(\uu_{s_1},\uv_{s_1})}\cdots x_{(\uu_{s_1},\uv_{s_1}   \\
{\footnotesize
\left(
\begin{array}{cccccccccc}
a_1x_{(\uu_{s_1}, \uv_{s_1})} & a_1y_{\uu_F}   & \cdots & 0 & 0  & \cdots &0 \\
%a_2 x_{(\uu_{s_2}, \uv_{s_2})} & 0 & a_2 x_{\uu_F}  & \cdots & 0 \\
\vdots \\
a_l x_{(\uu_{s_{\l}}, \uv_{s_{\l}})} & 0 &  \cdots & a_l y_{\uu_F} & 0 &  \cdots & 0\\
* & 0 & \cdots & 0 & b_1 & \cdots & 0 \\
\vdots \\
* & 0 & \cdots & 0 & 0 & \cdots & b_q \\
0 & \sgn (s_1)&  \cdots  & \sgn (s_{\l}) & 0 & \cdots & 0
%\sgn(t_1)e_1 & \cdots & \sgn (t_q)e_q
\end{array}
\right) (\bz),
}
\end{eqnarray}
which has full rank.\footnote{There are typos and small errors in \cite{Hu2025}.
For example, $\widetilde\Ga_\fV$ is used in pp 152, but only introduced in pp 156 and
its application is justified in pp 157.
Likewise, in pp 129, Lemma 7.3 (3), $Z_{[k], \Ga}$ ($Z^\dagger_{[k], \Ga}$) should be
$Z_{\sF_{[k]}, \Ga}$ ($Z^\dagger_{\sF_{[k]}, \Ga}$).}
Clearly, all the variables used to compute  $J^*(\fG_{\fV,F}|_{\tGa_\fV})$ are pleasant.

$\bullet$ Case $(\beta)$ as in Theorem \ref{wp-ell-blowup} (2).

In this case, the chart $\fV$ lies over $(x_{(\uu_{s_0},\uv_{s_0})} \ne 0)$ for some
$s_0 \in S_F \- \{s_F\}$.

We  introduce the following maximal minor of 
$\Jac (\fG_{\fV,F}|_{\tGa_\fV})$
$$J^*(\fG_{\fV,F}|_{\tGa_\fV})= {{\partial(B_{\fV, s_0}|_{\tGa_\fV},
B_{\fV, s_1}|_{\tGa_\fV} \cdots B_{\fV, s_{\l}}|_{\tGa_\fV}), B_{\fV, t_1}|_{\tGa_\fV} \cdots 
B_{\fV, t_q}|_{\tGa_\fV}, L_{\fV, F}|_{\tGa_\fV})}
 \over {{\partial(  y_{ \uu_F},
x_{(\uu_{s_1},\uv_{s_1})} \cdots x_{(\uu_{s_{\l}},\uv_{s_l})}}},
x_{(\uu_{t_1},\uv_{t_1})} \cdots x_{(\uu_{t_q},\uv_{t_q})},
 y_{\fV, (\um,\uu_F)})}. $$

Then, one computes and finds that at the point $\bz$, it is equal to
\begin{eqnarray} \nonumber
% J^*(\cB^\gov_{\fV,F}, L_{\fV,F}|_{\tGa_\fV}) =   
% y_{\uu_F}, x_{(\uu_{s_1},\uv_{s_1})}\cdots x_{(\uu_{s_1},\uv_{s_1}   \\
\left(
\begin{array}{cccccccccc}
 a_0 & 0 & \cdots   & 0  & 0   & \cdots & 0 & 0\\
*  &  a_1  & \cdots &   0  & 0 & \cdots & 0 & 0\\
%a_2 x_{(\uu_{s_2}, \uv_{s_2})} & 0 & a_2 y_{\uu_F}  & \cdots & 0 \\
\vdots \\
*  & 0 & \cdots & a_l   & 0   & \cdots & 0 & 0       \\
*  & 0&  \cdots & 0 & b_1 & \cdots  &0     & 0\\
%*  &  b_1  & \cdots &   0  & 0 \\
%a_2 x_{(\uu_{s_2}, \uv_{s_2})} & 0 & a_2 y_{\uu_F}  & \cdots & 0 \\
\vdots \\
*  &  0 & \cdots & 0 & 0 & \cdots & b_l         & 0     \\
 0 & * & \cdots & *&  * & \cdots  & * &   \sgn (s_F) 
\end{array}
\right) (\bz).
\end{eqnarray}
Thus, we conclude that 
$J^*(\cB^\gov_{\fV,F}|_{\tGa_\fV})$ %\red L_{\fV,F}|_{\tGa_\fV})$ 
is a square matrix of full rank at $\bz$, and all the variables  used to compute it are pleasant.

Combining ($\alpha$) and ($\beta$), we obtain
the  maximal minor of  $\Jac (\fG_\fV|_{\tGa_\fV})$
% \begin{equation}\label{the-grand-matrix}  \red
%J^*(\fG_\fV|_{\tGa_\fV})=\left(
%\begin{array}{cccccccccc}
% J^*(\fG^\star_{\fV,F_1}|_{\tGa_\fV})  & 0 & \cdots & 0  \\
% * & J^*(\fG_{\fV,F_2}|_{\tGa_\fV}) &  0 & \cdots & 0  \\
%\vdots &    \\
% * &  \cdots & J^*(\fG^\star_{\fV, F_\up}|_{\tGa_\fV}) & 0\\
% * & \cdots & * & J^*(\fL^{\lt, \inc})
% \end{array}
%\right).
%\end{equation}
 \begin{equation}\label{the-grand-matrix}  
J^*(\fG_\fV|_{\tGa_\fV})=\left(
\begin{array}{cccccccccc}
 J^*(\fG_{\fV,F_1}|_{\tGa_\fV})  & 0 & 0& \cdots & 0  \\
 * & J^*(\fG_{\fV,F_2}|_{\tGa_\fV}) &  0 & \cdots & 0  \\
\vdots &    \\
 * &  * & * & \cdots & J^*(\fG_{\fV, F_\up}|_{\tGa_\fV}) \\
 \end{array}
\right)
\end{equation}
where
all the blocks along diagonal are invertible at $\bz$.

First, take $\Ga=\emptyset$, \eqref{the-grand-matrix} implies that
$\tsV_\ell$ is smooth, and is  defined  by 
$$L_{\fV,\sF}, \; \cB_\fV^\gov$$ over any chart 
$\fV$, implying  all the relations in $\cB^\ngv$ are dependent and can be discarded.

Then, \eqref{the-grand-matrix} implies that $\tZ_{\ell, \Ga}$ is smooth, as a local complete 
intersection over $\fV$, defined by $\tGa_\fV, L_{\fV, \sF}$, and $\cB_\fV^\gov$.
\end{proof}

\end{document}